\documentclass[a4paper]{amsart}
\pdfoutput=1
\usepackage{amssymb}
\usepackage{mathtools}
\usepackage{graphicx}
\usepackage{caption}
\usepackage{subcaption}
\usepackage{xcolor}

\newcommand{\refchange}[1]{{\color{black}{#1}\color{black}}{}}
\newcommand{\otherchange}[1]{\color{black}{#1}\color{black}{}}

\usepackage[colorlinks,
    citecolor=green,
    filecolor=black,
    linkcolor=blue,
    urlcolor=black
]{hyperref}
\hypersetup{linktocpage}
\usepackage{tikz}
\usetikzlibrary{
    decorations.markings
    ,calc
}

\newtheorem{theorem}{Theorem}[section]
\newtheorem{corollary}[theorem]{Corollary}
\newtheorem{lemma}[theorem]{Lemma}
\newtheorem{proposition}[theorem]{Proposition}

\theoremstyle{definition}
\newtheorem{definition}[theorem]{Definition}

\newtheorem{remark}[theorem]{Remark}

\newcommand{\simsub}[1]{\underset{#1}{\sim}}

\title{Regular isotopy classes of link diagrams from Thompson's groups}
\author{Rushil Raghavan}
\author{Dennis Sweeney}

\dedicatory{To the memory of Vaughan Jones.}

\begin{document}

\maketitle

\begin{abstract}
In 2014, Vaughan Jones developed a method to produce links from elements of Thompson's group $F$, and showed that all links arise this way. He also introduced a subgroup $\vec{F}$ of $F$ and a method to produce oriented links from elements of this subgroup. In 2018, Valeriano Aiello showed that all oriented links arise from this construction. We classify exactly those regular isotopy classes of links that arise from $F$, as well as exactly those regular isotopy classes of oriented links that arise from $\vec{F}$, answering a question asked by Jones in 2018.
\end{abstract}

\tableofcontents

\section{Introduction}
In \cite{jones2014unitary}, while developing a method to construct unitary representations of Thompson's group $F$, Jones introduced a construction of knots and links from $F$ as a byproduct of certain unitary representations of $F$. He then showed, in analogy with Alexander's Theorem for producing links from braids, that any link diagram is equivalent to the link diagram given by an element of $F$. He also introduced a subgroup $\vec{F}$ of $F$ called the oriented Thompson group and a construction of oriented links from elements of $\vec{F}$. Later, Aiello \cite{aiello2020} proved that an anologue to Alexander's Theorem holds for $\vec{F}$ as well, showing that every oriented link diagram is equivalent to some link diagram arising from $\vec{F}$. Jones asked in \cite[Section 7, Question 3]{jones2018construction} whether all regular isotopy classes of link diagrams arise from elements of $F$. We answer this question in the negative, and classify exactly which regular isotopy classes of link diagrams arise from elements of $F$. Moreover, we classify exactly which regular isotopy classes of oriented links arise from elements of $\vec{F}$. 

\section{Preliminary Definitions}

\begin{definition} \label{equivalentdef}
Link diagrams $D_1$ and $D_2$ are said to be \textit{equivalent} if they give the same link, or equivalently (from \cite{Reidemeister}) if they are reachable from each other in finitely many Reidemeister moves. In this case, we will write $D_1 \simsub{123} D_2$.
\end{definition}

\begin{definition} \label{regisodef}
Link diagrams $D_1$ and $D_2$ are said to be \textit{regular-isotopic} if they are reachable from each other in finitely many second and third Reidemeister moves, forbidding the first Reidemeister move. In this case, we will write $D_1 \simsub{23} D_2$.
\end{definition}

The notion of regular isotopy of link diagrams was introduced by L. Kauffman in \cite{Kauffman}. This spurred the development of many important link invariants, such as the Kauffman bracket, and regular isotopy invariants of links have since become a common object of study, for instance in \cite{KauffmanInvariant}. In \cite{jones2018construction}, Section 3, Jones discusses a skein theory variant of his construction of links from elements of Thompson's groups. In this variant, regular isotopy invariants of link diagrams arise as coefficients of representations of $F$. 

\begin{definition} \label{whitneydef}
The \textit{Whitney index} $\Omega(K)$ of an oriented knot diagram $K$ will refer to the total curvature of an immersed plane curve \otherchange{divided by $2\pi$} (also called the turning number---the winding number of the derivative vector around the origin), which by the result of \cite{whitney} may be computed for a knot diagram as follows: we choose a base point $b$ on an arc of $K$ that is adjacent to the unbounded component of the complement of $K$. Then we traverse $K$ in the assigned direction starting at $b$, labeling the first and second directions through which a crossing is traversed with a $1$ and a $2$ respectively. Then we can compute the Whitney index as: 
$$
\includegraphics{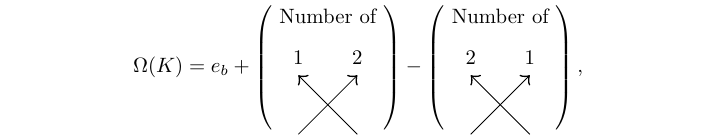}
$$
where $e_b$ depends on the orientation of $K$ at $b$: if the traversal at $b$ has the unbounded component to its right, $e_b=+1$; otherwise, $e_b=-1$. Note that the Whitney index does not depend on the assignment of an over-strand and under-strand at each crossing. When computing the Whitney index of one component $C$ of a link diagram, only the crossings of $C$ with itself are counted; crossings between components are ignored.
\end{definition}

\begin{definition} \label{writhedef}
The \textit{writhe} $w(K)$ of a knot diagram $K$ will refer to the sum of the \textit{signs} of the crossings in $K$, assigned as follows:
$$
\includegraphics{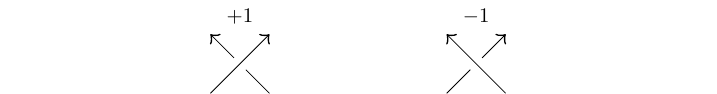}
$$
Again, when computing the writhe of one component $C$ of a link diagram, the sum is only taken over crossings of $C$ with itself; crossings between components are ignored.
\end{definition}

\begin{lemma}[From \cite{coward}] \label{cowardlemma}
A pair $\{D_1, D_2\}$ of oriented link diagrams has $D_1 \simsub{23} D_2$ if and only if both of the following hold:
\begin{enumerate}
    \item $D_1 \simsub{123} D_2$
    \item For each component $C_1$ of $D_1$ and the corresponding component $C_2$ of $D_2$, we have $\Omega(C_1)=\Omega(C_2)$ and $w(C_1)=w(C_2)$.
\end{enumerate}
Above, ``corresponding component'' refers to the correspondence between the components of $D_1$ and $D_2$ defined by the Reidemeister moves implicit in (1).
\end{lemma}

\refchange{
\begin{definition}\label{Thompson} Thompson's group $F$ is the group of piecewise-linear, orientation-preserving self-homeomorphisms of $[0,1]$, differentiable everywhere except at finitely many dyadic rational numbers, with derivative a power of $2$ wherever the function is differentiable. 

An element of $F$ can also be depicted as a pair of rooted binary trees, each with the same number of leaves. For each element of $F$ there is a unique ``reduced'' representation as a pair of rooted binary trees, that is, a representation with a minimal number of leaves. We will draw a pair of trees with one tree upside down on top of the other and join their leaves; see the left side of Figure \ref{fig:Ldef} for an example. 

For a more detailed introduction to Thompson's group $F$, see \cite{CFP}.
\end{definition}
}

\section{A classification of the regular isotopy classes in \texorpdfstring{$L(F)$}{L(F)}}
\begin{definition}\label{Ldef}
Let $L$ be the mapping defined in \cite{jones2014unitary} from Thompson's group $F$ to the set of unoriented link diagrams. In this mapping, a reduced pair-of-trees representation of an element of $F$ \refchange{with $n$ leaves} is transformed into a link diagram by adding,
\refchange{for each $i \in \{1, \dots, n-1\}$, an arc connecting the the $i$th non-leaf node (counting left-to-right) of the top tree to the $i$th non-leaf node of the bottom tree. We will refer to such arcs as $\textit{supplementary arcs}$.
Each non-leaf node then becomes a crossing, where a strand connecting that node's left and right descendants passes over a strand connecting the supplementary arc to the node's parent.}
Finally, a \textit{closure strand} is added that connects the roots of the two trees around the left side of the diagram.
\refchange{For better understanding the construction, the reader is referred to \cite{jones2018construction}.}
For an example, see Figure \ref{fig:Ldef}.
\begin{figure}
    \centering
    \includegraphics{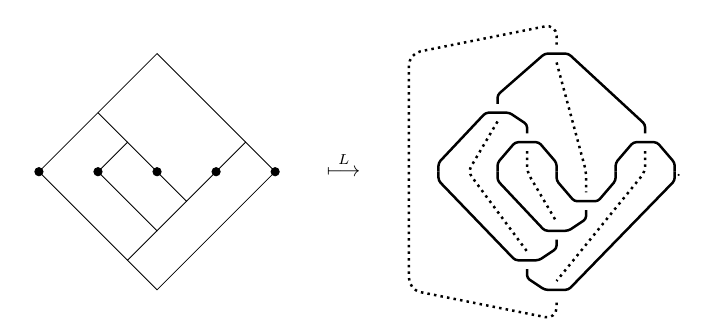}
    \caption{An example of the mapping $L$.}
    \label{fig:Ldef}
\end{figure}
\end{definition}

\begin{lemma}[Theorem 3.0.1 in \cite{jones2018construction}] \label{jonesthompsonlemma}
For any unoriented link diagram $D$, there exists $g\in F$ such that $L(g)\simsub{123}D$.
\end{lemma}

\refchange{
\begin{definition}
    Given some $g \in F$, a \textit{leaf-arc} in $L(g)$
    will refer a small arc in $L(g)$ at the position of a leaf in the reduced pair-of-trees representation of $g$.
    In images like Figure \ref{fig:Ldef}, leaf-arcs
    are solid (not dotted) arcs that cross the central horizontal axis. If the intersections of strands with the central
    horizontal axis are labelled $1, \dots, 2n$, then
    the leaf-arcs are found at the even-numbered intersections
    (the closure strand and supplementary arcs contain the odd-numbered intersections).
\end{definition}
}

\begin{lemma}\label{leaflemma}
For any unoriented link diagram $D$, there is a $g \in F$ such that $L(g)\simsub{123}D$ and such that each component $C$ of $L(g)$ includes some \refchange{leaf-arc.}
\end{lemma}

\begin{proof}
Let $D$ be a link diagram, and by Lemma \ref{jonesthompsonlemma}, let $g_0\in F$ satisfy $L(g_0) \simsub{123} D$. It suffices to manipulate $g_0$ to find some $g\in F$ such that $L(g)\simsub{123}L(g_0)$, but such that each component of $L(g)$ has a leaf-arc.

Suppose $L(g_0)$ has $m$ components without leaf-arcs. Assuming $m\geqslant 1$, we show that there is some $g_1$ with $L(g_1)\simsub{123}L(g_0)$, but with $L(g_1)$ having only $m-1$ components without leaf-arcs; the rest follows by induction,
\refchange{since every $L(g)$ has a leaf-arc in some component}.

Let $C$ be a component of $L(g_0)$ with no leaf-arc. Since $C$ must pass between the upper and lower halves of the diagram and does not do so at a leaf-arc, it must do so at a \refchange{supplementary arc}. We can therefore let $g'$ be the result of applying the move described by Figures \ref{fig:makeleaftree}, \ref{fig:makeleaf_slowmo}, \ref{fig:diagram_before_makeleaf}, and \ref{fig:diagram_after_makeleaf} to $g_0$, in which the arc above the top \refchange{node} of a \refchange{supplementary} arc of $C$ is manipulated with Reidemeister II moves so as to intersect with its leftmost descendent leaf-arc and obtain 3 new leaf-arcs. Let $B$ be the component of $L(g_0)$ containing that leftmost descendent \refchange{leaf-arc}. Since this old leaf-arc of $B$ in $L(g_0)$ is no longer a leaf-arc of $B$ in $L(g')$, we may need to make one more manipulation: applying the same move to $B$ will leave all other components unaffected, except that $C$ will lose 1 leaf-arc, but this is unimportant because $C$ had already gained 3 leaf-arcs from the previous move. Thus, we have found $g_1\in F$, which as a result of the described move potentially iterated twice, makes $L(g_1)$ have exactly one more component with a leaf-arc.
\end{proof}

\begin{figure}
    \centering
    \includegraphics{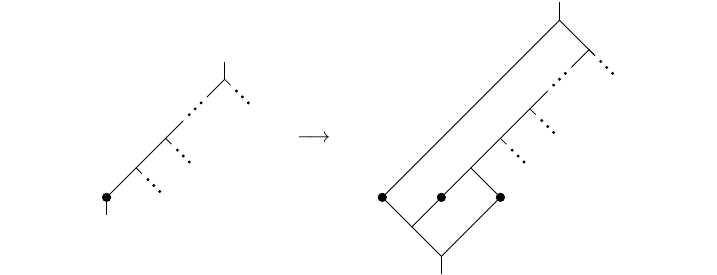}
    \caption{A manipulation on Thompson's group elements that ensures that the link is preserved, but that a particular component of the resulting link passes through a leaf-arc of the resulting diagram}
    \label{fig:makeleaftree}
\end{figure}

\begin{figure}
    \centering
    \raisebox{-0.5\height}{\includegraphics{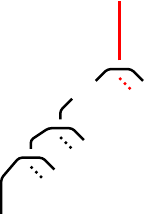}}
    \raisebox{-0.5\height}{\includegraphics{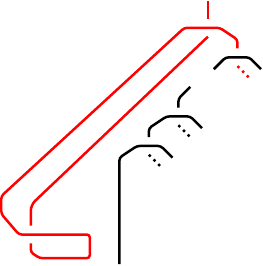}}
    \raisebox{-0.5\height}{\includegraphics{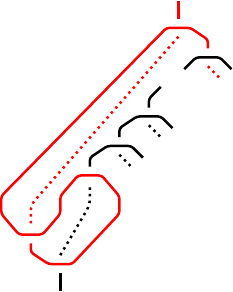}}
    \caption{\refchange{The manipulation of trees in Figure~\ref{fig:makeleaftree}
    corresponds to two Reidemeister II moves on the corresponding link diagrams.}}
    \label{fig:makeleaf_slowmo}
\end{figure}

\begin{figure}
    \centering
    \includegraphics{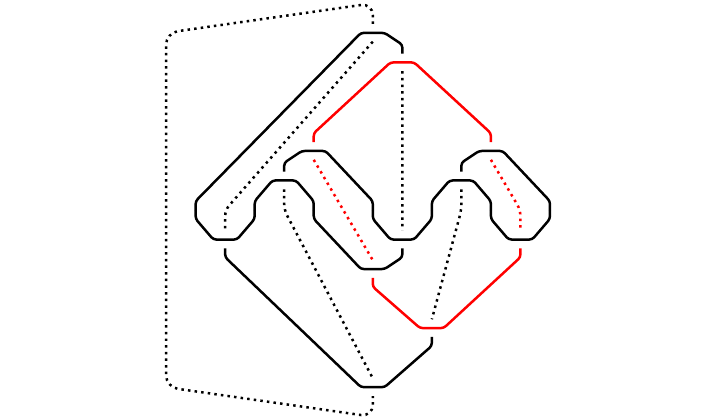}
    \includegraphics{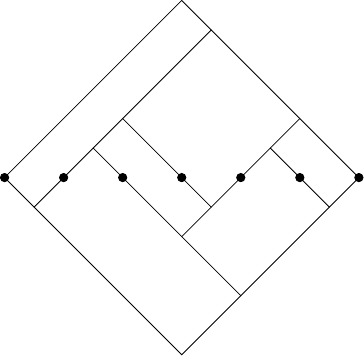}
    \caption{A link diagram where one component (drawn in red) has no leaf-arc, and the corresponding element of Thompson's group}
    \label{fig:diagram_before_makeleaf}
\end{figure}

\begin{figure}
    \centering
    \includegraphics{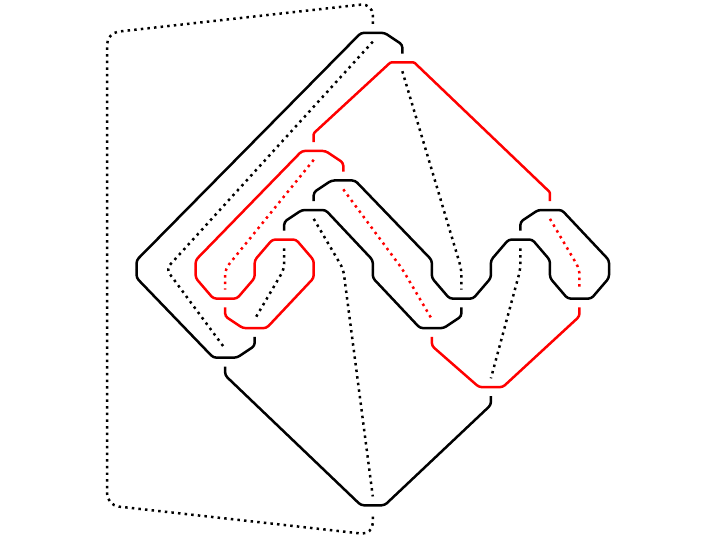}
    \includegraphics{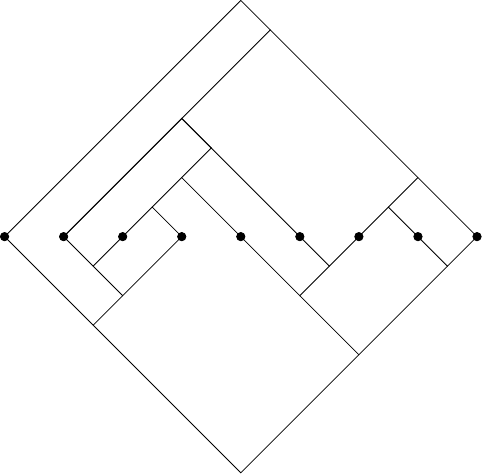}
    \caption{A link diagram regular-isotopic to Figure \ref{fig:diagram_before_makeleaf} but with the red component having a leaf-arc, and the corresponding element of Thompson's group}
    \label{fig:diagram_after_makeleaf}
\end{figure}

\begin{definition}
For a link diagram $D$ with components $C_1, \dots, C_n$, we will define three functions $$
    \alpha_D, \beta_D, \gamma_D \colon \{C_1, \dots, C_n\} \to \mathbb{Z}/2\mathbb{Z}.
$$
For each $i$, let $\alpha_D(C_i)$ be the parity of the the number of crossings of $D$ where $C_i$ is both the under-strand and the over-strand. Let $\beta_D(C_i)$ be the parity of the number of crossings of $D$ for which $C_i$ is the under-strand, but $C_j$ is the over-strand for some $j\neq i$. Let $\gamma_D = \alpha_D + \beta_D$, so that for each $i$, $\gamma_D(C_i)$ is the parity of the total number of crossings in which $C_i$ is the under-strand.
\end{definition}

\begin{proposition} \label{parityprop}
If some component $C$ of a link diagram has $\alpha(C)=0$, then $w(C)$ is even and $\Omega(C)$ is odd. If $\alpha(C) = 1$, then $w(C)$ is odd and $\Omega(C)$ is even.
\end{proposition}
\refchange{
\begin{proof}
    This parity argument follows directly from
    Definitions \ref{whitneydef} and \ref{writhedef}.
\end{proof}
}

\begin{proposition} \label{abcinvariantprop}
$\gamma_D$ and $\alpha_D$ are regular-isotopy invariants, and $\beta_D$ is invariant under all three Reidemeister moves.
\end{proposition}
\refchange{
\begin{proof}
Let $C$ be a component of $D$. A Reidemeister II move will either not change the number of crossings for which $C$ is the under-strand, or change that number by $2$, so $\gamma_D$ is invariant under Reidemeister II moves. Similarly, a Reidemeister II move will either not change the number of crossings for which $C$ is both the over-strand and under-strand, or it will change the number of such crossings by $2$, so $\alpha_D$ is invariant under Reidemeister II moves. 

Reidemeister III moves do not change the number of crossings, 
nor do they change the component containing the over-strand or under-strand of any given
crossing, so Reidemeister III moves preserve $\alpha_D$, $\beta_D$ and $\gamma_D$.

Since $\gamma_D$ and $\alpha_D$ are regular isotopy invariant, so is $\beta_D = \gamma_D - \alpha_D$. It remains to show that $\beta_D$ is invariant under Reidemeister I moves. 
A Reidemeister I move does not remove or add crossings between two different components, so $\beta_D$ is unchanged by a Reidemeister I move. 
\end{proof}
}

\begin{definition} \label{compliantdef}
Call a link diagram $D$ \textit{compliant} if $\gamma_D$ is identically zero, or equivalently, if $\alpha_D = \beta_D$. See Figure \ref{fig:compliantandnon} for an example and a non-example.

\end{definition}

\begin{figure}
    \centering
    \begin{subfigure}{.5\textwidth}
        \centering
        \includegraphics[scale=0.8]{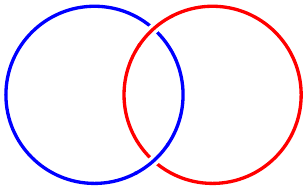}
        \caption{A non-compliant Hopf link}
        \label{fig:noncompliant}
    \end{subfigure}%
    \begin{subfigure}{.5\textwidth}
        \centering
        \includegraphics[scale=0.8]{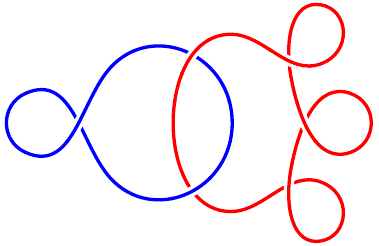}
        \caption{A compliant Hopf link}
        \label{fig:compliant}
    \end{subfigure}
    \caption{Two diagrams of the Hopf link that are not regular-isotopic.\refchange{
In Figure $\ref{fig:noncompliant}$,
$\alpha_D$ of each component is $0$,
$\beta_D$ of each component is $1$,
and $\gamma_D$ of each component is $1$.
In Figure $\ref{fig:compliant}$,
$\alpha_D$ of each component is $1$,
$\beta_D$ of each component is $1$,
and $\gamma_D$ of each component is $0$.
}}
    \label{fig:compliantandnon}
\end{figure}

\begin{theorem} \label{mainthm}
A link diagram $D$ is compliant if and only if there is some element $g\in F$ for which $L(g)\simsub{23} D$.
\end{theorem}

\begin{proof}
First suppose that $D$ is a link diagram with some $L(g)\simsub{23} D$.
Since $\gamma_D$ is regular-isotopy invariant, to show that $D$ is compliant, it suffices to show that $L(g)$ is compliant. Let $C$ be a component of $L(g)$. By the construction for $L$, each crossing in the \textit{top} half-plane that has $C$ as the under-strand corresponds bijectively to a crossing in the \textit{bottom} half-plane that has $C$ as its under-strand, \otherchange{using the bijection formed by supplementary arcs}. It follows that there are an even number of such crossings, so $\gamma_D(C) = 0$. This holds for each component $C$, so $L(g)$ is compliant.

Now suppose that $D$ is a compliant link diagram. By Lemma~\ref{leaflemma}, there is some $g\in F$ with $L(g) \simsub{123} D$, and with each component of $L(g)$ having a leaf-arc. Introduce an orientation on each component of $D$, and by the equivalence of the link diagrams, introduce the corresponding orientations on the components of $L(g)$. By Lemma~\ref{cowardlemma}, it suffices to manipulate $g$ (replacing it with some other $g'\in F$) so that for each component $C$ of $D$, the corresponding component $C'$ of $L(\refchange{g'})$ has $\Omega(C)=\Omega(C')$ and $w(C) = w(C')$, while retaining $L(\refchange{g'})\simsub{123}D$. We accomplish this by adjusting $\Omega$ and $w$ for each component individually.

Let $C$ be a component of $D$ and let $C'$ be the corresponding component of $L(g)$. Since $D$ and $L(g)$ are equivalent, and each is compliant, we have
$$
    \alpha_{D}(C) = \beta_D(C) = \beta_{L(g)}(C') = \alpha_{L(g)}(C').
$$
By Proposition~\ref{parityprop}, it follows that $w(C')-w(C)$ and $\Omega(C')-\Omega(C)$ are even. Thus, it suffices to repeatedly add $\pm 2$ to $w(C')$ and $\Omega(C')$ until they match $w(C)$ and $\Omega(C)$. This adjustment can be accomplished by replacing a leaf-arc of $L(g)$ with any of the diagrams in Figure \ref{fig:unoriented-insertions}, as necessary.
\end{proof}

\begin{figure}
    \centering
    \includegraphics{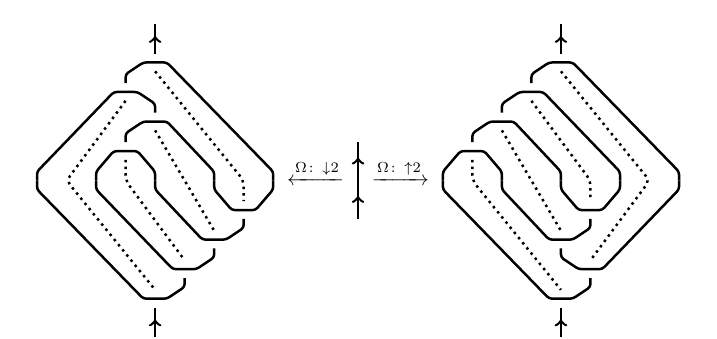}
    \includegraphics{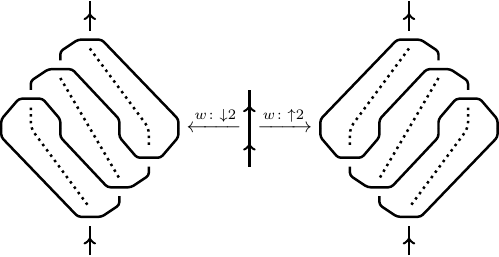}
    \caption{Replacing a leaf-arc with one of the four depicted diagrams results in either the writhe $w$ or the index $\Omega$ of the component containing that leaf-arc increasing or decreasing by $2$. Each move preserves the link. If the leaf-arc is directed down rather than up then the \otherchange{$\Omega$} moves have the inverse effects.}
    \label{fig:unoriented-insertions}
\end{figure}

\section{A classification of the regular isotopy classes in \texorpdfstring{$L_3(F_3)$}{L3(F3)}}

\refchange{
\begin{definition}
    \label{L3def}
    Let $F_3$ be the \textit{ternary Thompson's group}
    defined in \cite{jones2018construction},
    and for each $h \in F_3$ define $L_3(h)$
    to be the link diagram arising from $h$
    as defined in \cite{jones2018construction}.
    This construction is similar to Definition \ref{fig:Ldef},
    except there is no need for supplementary arcs;
    after adding the closure strand,
    all non-leaf vertices are already 4-valent.
\end{definition}

\begin{corollary}
    \label{corollaryf3}
    A link diagram $D$ is compliant if and only if there
    is some $h \in F_3$ such that $L_3(h) \simsub{23} D$.
\end{corollary}
\begin{proof}
    First suppose $D$ is a compliant link diagram.
    Then by Theorem \ref{mainthm}, there is some $g \in F$
    such that $L(g) \simsub{23} D$.
    In \cite[Section 4]{jones2018construction},
    Jones shows that there is a natural embedding
    $\phi\colon F\to F_3$ such that $L_3 \circ \phi = L$.
    Therefore $L_3(\phi(g)) \simsub{23} D$,
    so the desired result holds with $h = \phi(g)$.

    For the converse direction, let $h \in F_3$.
    Let $C$ be a component of $L_3(h)$.
    For each component $B\neq C$
    of $L_3(h)$, by a general fact about plane curves,
    there are an even number of crossings of
    $B$ with $C$, so $C$ has an even number of total crossings
    with other components.
    Let $\delta \in \mathbb{Z}/2\mathbb{Z}$ be the parity
    of the number of crossings in which $C$ is the over-strand
    but not the under-strand.
    Counting all intersections of $C$ with other components,
    we have that $\delta + \beta_{L_3(h)}(C) = 0$, so $\delta = \beta_{L_3(h)}(C)$.
    
    By the construction of $L_3(h)$,
    traversing $C$ in some direction,
    we see that the crossings in which $C$ is the over-strand
    alternate between the top and bottom half-planes:
    each such crossing occurs when $C$
    has a local maximum or minimum in the vertical direction,
    with maxima in the upper half-plane and minima in the lower half-plane. We conclude that there are an even number of
    crossings in which $C$ is the over-strand.
    Counting by whether the under-strand is $C$ as well
    gives $\alpha_{L_3(h)}(C) + \delta = 0$.
    We thus have
    \[
        \alpha_{L_3(h)}(C) = \delta = \beta_{L_3(h)}(C),
    \]
    and this holds for all components $C$, so $L_3(h)$ is compliant.
    Compliance is invariant under regular isotopy,
    so if $D$ is a link diagram with $D \simsub{23} L_3(h)$,
    then $D$ is compliant as well.
\end{proof}
}


\section{A classification of the regular isotopy classes in \texorpdfstring{$\vec{L}(\vec{F})$}{\vec{L}(\vec{F})}}

\begin{definition} \label{vecfdef}
The oriented Thompson group $\vec{F} < F$ is the subgroup of elements $g\in F$ for which \refchange{the} checkerboard-shading \refchange{of} $L(g)$ gives a Seifert surface, i.e., the graph of shaded faces connected across vertices is bipartite. See \cite[Definition 5.2.0.7.]{jones2014unitary}.
\end{definition}

\begin{definition} \label{vecldef}
Let $\vec{L}$ be the mapping defined in \cite{jones2014unitary} from the oriented Thompson's group $\vec{F}$ to the set of oriented link diagrams. For example, see Figure \ref{fig:vecLdef}.
\begin{figure}
    \centering
    \includegraphics{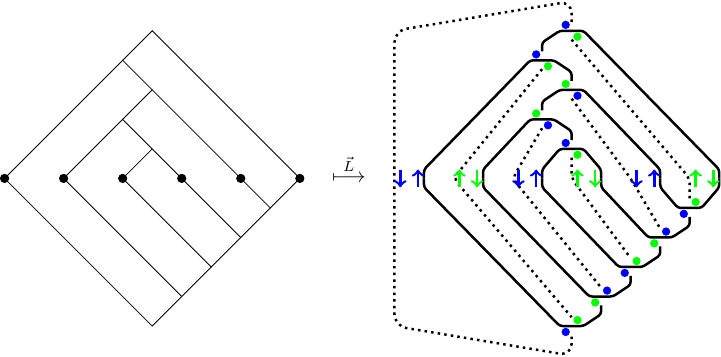}
    \caption{An example of the mapping from $\vec{F}$ to oriented links}
    \label{fig:vecLdef}
\end{figure}
For any $g\in\vec{F}$, $\vec{L}(g)$ is the same link diagram as $L(g)$ in Definition~\ref{Ldef}, but with a canonical orientation assigned to the strands. The faces marked with blue (including the region enclosed by the closure strand) should be surrounded by arcs directed counterclockwise around that face, and the green faces should have clockwise-directed arcs. Since the graph of colored faces connected at crossings is bipartite, this introduces a consistent orientation on the components of the link.
\end{definition}

\begin{lemma}[From \cite{aiello2020}] \label{aiellolemma}
Given an oriented link diagram $D$, there is an element $g\in \vec{F}$ such that $\vec{L}(g)\simsub{123}D$.
\end{lemma}

\begin{lemma}\label{orientedleaflemma}
For any oriented link diagram $D$, there is some $g\in\vec{F}$ such that $\vec{L}(g)\simsub{123}D$ and such that each component of $\vec{L}(g)$ has some leaf-arc.
\end{lemma}

\begin{proof}
The proof is identical to that of Lemma \ref{leaflemma} for the unoriented case, except a more complicated move is needed to maintain the orientedness of the diagram while ensuring that a component has a leaf-arc. This move is realized by Figure~\ref{fig:oriented-makeleaf}.
\refchange{
To verify that orientedness is preserved,
consider the checkerboard shading of some link diagram
that has some portion that looks like the diagram on the left
of Figure~\ref{fig:oriented-makeleaf}.
Replacing that portion with the diagram on the right
maintains that the outside faces on the right
have the opposite coloring as the face on the left.
These are the only shaded faces adjacent to this portion of the diagram,
and all newly added faces are the opposite color
to their neighbors (each blue dot is opposite a green dot and vice-versa), and so we still have a bipartition.
The case where blue and green are swapped is identical.
}
\end{proof}

\begin{figure}
    \centering
    \includegraphics{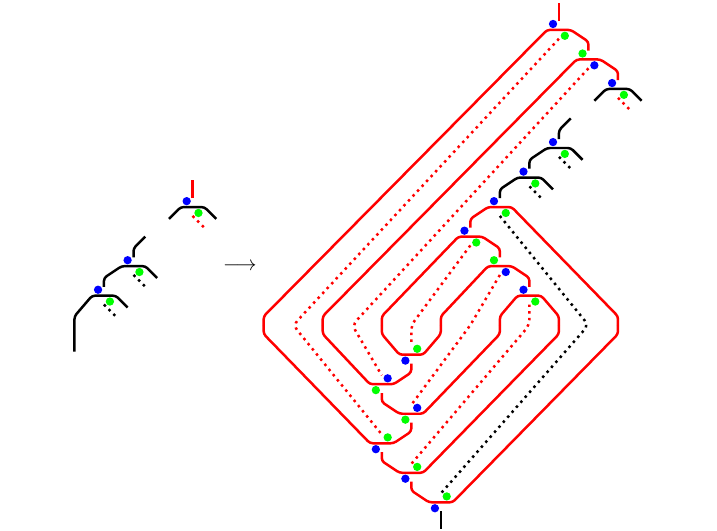}
    \caption{A transformation on link diagrams from $\vec{F}$ that preserves the link, but ensures that the red component goes through a leaf-arc. This move is similar to to the move shown in Figures \ref{fig:makeleaftree}, \ref{fig:makeleaf_slowmo}, \ref{fig:diagram_before_makeleaf}, and \ref{fig:diagram_after_makeleaf}, but this move is more complicated so as to preserve orientedness.}
    \label{fig:oriented-makeleaf}
\end{figure}

\begin{remark}
In the theorem below, we show that the condition required for a regular isotopy class of links to be represented by $\vec{F}$ is strictly stronger than the condition for $F$. The proof follows the structure of Theorem \ref{mainthm}, but we will work in $\mathbb{Z}$ rather than $\mathbb{Z}/2\mathbb{Z}$.
\end{remark}

\begin{definition}
For an oriented link diagram $D$ with components $C_1, \dots, C_n$, we will define three functions
$$
    w_D, cw_D, u_D\colon \{C_1, \dots, C_n\} \to \mathbb{Z}.
$$
For each $i$, let $w_D(C_i)=w(C_i)$ be the sum of the signs of the crossings of $D$ where $C_i$ is both the under-strand and the over-strand. Let $cw_D(C_i)$ be the sum of the signs of the crossings of $D$ for which $C_i$ is the under-strand but $C_j$ is the over-strand for some $j\neq i$. Let $u_D = w_D + cw_D$, so that for each $i$, $u_D(C_i)$ is the sum of the signs of all crossings where the under-strand is $C_i$.
\end{definition}

\begin{proposition} \label{orparityprop}
For any component $C$ of a link diagram $D$,
the integers $w_D(C)$ and $\Omega(C)$ have opposite parities.
\end{proposition}
\refchange{\begin{proof}
This is clear by Definitions \ref{whitneydef} and \ref{writhedef}.
\end{proof}}

\begin{proposition} \label{orinvariantprop}
$w_D$ and $u_{\otherchange{D}}$ are regular-isotopy-invariant, and $cw_{\otherchange{D}}$ is invariant under all three Reidemeister moves.
\end{proposition}
\otherchange{
\begin{proof}
    Whether the two involved strands are directed in the same
    or opposite directions, Reidemeister II moves add or remove
    two crossings with opposite signs.
    If the two strands are part of the same component,
    then $w_D$ for that component is unaffected
    because the two signs cancel,
    while $w_D$ for all other components clearly remains unchanged.
    In this case $cw_D$ is unchanged as well because
    there is no change to crossings between components.
    On the other hand, if the two strands involved
    in a Reidemeister II move are distinct,
    then $cw_D$ is unchanged for the under-strand component
    because the two signs cancel,
    and $cw_D$ is clearly unchanged for other components.
    In this case $w_D$ is unchanged as well
    because only crossings between components are affected.
    Since Reidemeister II preserves $w_D$ and $cw_D$,
    it preserves $u_D = cw_D + w_D$ as well.

    Reidemeister III moves do not affect the signs
    of crossings nor which component is the under-strand
    or over-strand, so $cw_D$, $w_D$ and $u_D$
    are preserved by Reidemeister III moves.
    
    Lastly, $cw_D$ is unaffected by Reidemeister I moves because
    Reidemeister I moves only affect the set of
    crossings within one component.
\end{proof}
}

\begin{definition} \label{orcompdef}
Call an oriented link diagram $D$ \textit{oriented-compliant} if $u_D$ is identically zero, or equivalently, if $w_D = -cw_D$. See Figure \ref{fig:orcompliantandnon} for an example and non-example\refchange{: in the left diagram, $u_D = (2,-2)$, so it is not oriented-compliant, while in the right diagram, $u_D = (0,0)$, so it is oriented-compliant.}
\end{definition}

\begin{remark}
Oriented-compliance is closely related to the \textit{linking matrix}. \refchange{If the link diagram $D$ has components $C_1,\dots,C_n$, then the $(i,j)$th entry of the linking matrix of $D$ is one half of the sum of the signs of the crossings between $C_i$ and $C_j$,
which is equal to the sum of
the signs of the crossings in which
$C_i$ is the under-strand and $C_j$
is the over-strand.
}
If we replace,  \refchange{for all $i$}, the \refchange{$i$th} diagonal entry with the writhe of \refchange{$C_i$}, then the oriented-compliance of $D$ is equivalent to each row of the resulting symmetric matrix summing to zero.
\end{remark}

\begin{figure}
    \centering
    \includegraphics[scale=0.8]{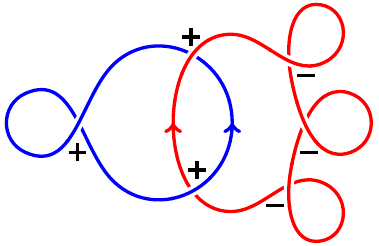}
    \hspace{1cm}
    \includegraphics[scale=0.8]{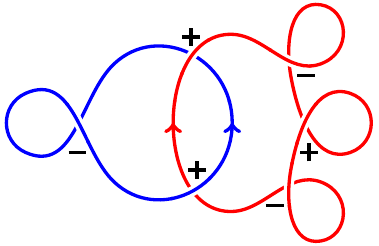}
    \caption{Shown are two diagrams of the Hopf link that are not regular-isotopic. The left diagram is compliant but not oriented-compliant, while the right diagram is both compliant and oriented-compliant.}
    \label{fig:orcompliantandnon}
\end{figure}

\begin{theorem} \label{orientedmainthm}
An oriented link diagram $D$ is \textit{oriented-compliant} if and only if there is some $g\in \vec{F}$ such that $\vec{L}(g) \simsub{23} D$.
\end{theorem}

\begin{proof}
First suppose that $D$ is an oriented link diagram with some $g \in \vec{F}$ such that $\vec{L}(g) \simsub{23} D$. Since $u_D$ is regular-isotopy invariant, to show that $D$ is oriented-compliant, it suffices to show that $\vec{L}(g)$ is oriented-compliant. Let $C$ be a component of $\vec{L}(g)$. By the construction for $\vec{L}$, each crossing in the top half-plane bijectively corresponds to a crossing in the bottom half-plane, and since the shaded face between these two crossings is consistently oriented either clockwise or counterclockwise, the two crossings have opposite signs. It follows that $u_D(C)$ vanishes. This holds for each component $C$, so $D$ is oriented-compliant.

Now suppose that $D$ is an oriented-compliant oriented link diagram. By Lemma \ref{orientedleaflemma}, there is some $g\in \vec{F}$ with $\vec{L}(g) \simsub{123} D$ and with each component of $\vec{L}(g)$ having a leaf-arc. By Lemma~\ref{cowardlemma}, it suffices to manipulate $g$ (replacing it by some other $g'\in\vec F$) so that for each component $C$ of $D$ and corresponding component $C'$ of $\vec{L}(\refchange{g'})$, we have $\Omega(C)=\Omega(C')$ and $w(C) = w(C')$, while maintaining $\vec{L}(\refchange{g'})\simsub{123} D$.

Let $C$ be a component of $D$ and let $C'$ be the corresponding component of $L(g)$.
Since $cw$ is invariant under all Reidemeister moves and since $D$ and $\vec{L}(g)$ are oriented-compliant, we already have
$$
    w(C)=w_D(C) = - cw_D(C) = - cw_{\vec{L}(g)}(C') = w_{\vec{L}(g)}(C') = w(C'),
$$
so we only need to correct the Whitney indices of $\vec{L}(g)$ to match those of $D$. By Proposition~\ref{orparityprop}, $\Omega(C') \equiv w(C')+1 = w(C)+1 \equiv \Omega(C) \pmod 2$.
It thus suffices to adjust the Whitney index of $C'$ by $\pm 2$ at a time. This can be achieved by using one of the moves shown in Figure \ref{fig:oriented-insertions}, in which one of the leaf-arcs of $C'$ is replaced by a different diagram. This move preserves orientedness, so we can freely adjust the Whitney index of $C'$ by multiples of 2 until $\Omega(C')=\Omega(C)$.

After repeating this process of adjusting $\Omega$ for each component $C$ of $L(g)$, we have that $D$ and $L(g)$ satisfy the suppositions of Lemma \ref{cowardlemma}, so $D\simsub{23} L(g)$.
\end{proof}

\begin{figure}
    \centering
    \includegraphics{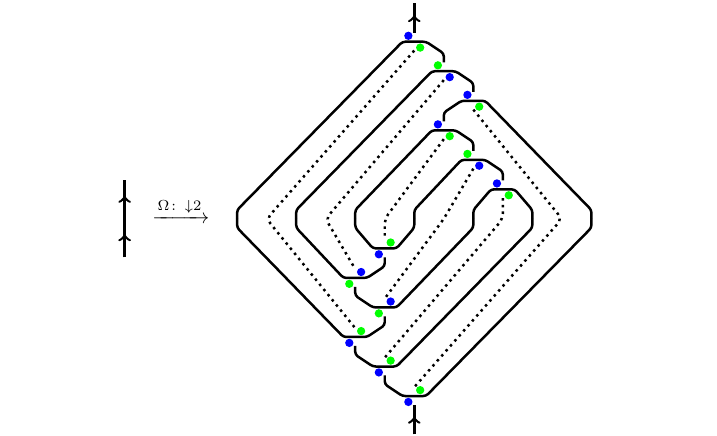}
    \includegraphics{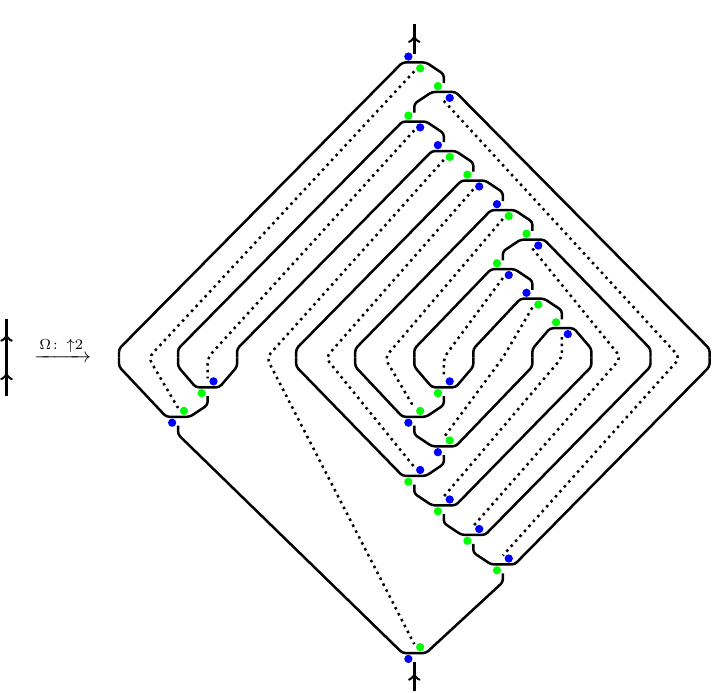}
    \caption{These two diagrams can be inserted at any leaf-arc of an existing diagram from $\vec{F}$ to get another oriented diagram with the index $\Omega$ changed by $\pm 2$. If the leaf-arc is directed down instead of up, these same moves apply, but they have the inverse effects.}
    \label{fig:oriented-insertions}
\end{figure}

\otherchange{
\section{A classification of the regular isotopy classes in \texorpdfstring{$\vec{L_3}(\vec{F_3})$}{\vec{L_3}(\vec{F_3})}}

\begin{definition}
    The oriented ternary Thompson group $\vec{F_3} < F_3$
    is the subgroup of elements $h \in F_3$
    for which the checkerboard-shading of $L_3(h)$
    (see Defintion~\ref{L3def}) gives a Seifert surface,
    i.e., the graph of shaded faces connected across vertices is bipartite.
    See Figure~\ref{fig:oreinted-three-mapping} for an example.
\end{definition}

\begin{definition}
    Let $\vec{L_3}$ be the mapping from $\vec{F_3}$
    to oriented link diagrams, as shown in Figure~\ref{fig:oreinted-three-mapping},
    similar to definition \ref{vecldef}.
\end{definition}

\begin{figure}
    \centering
    \includegraphics{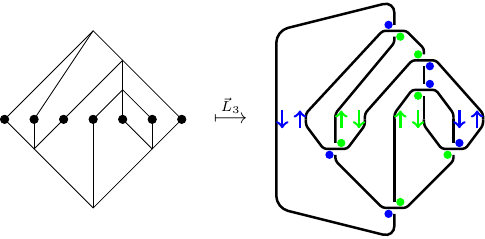}
    \caption{An example of the mapping from $\vec{F_3}$ to oriented links}
    \label{fig:oreinted-three-mapping}
\end{figure}

\begin{remark}
    Despite the similarity of the definitions of
    $\vec{L_3}$ and $\vec{L}$ (Definition~\ref{vecldef}),
    the theorem below shows that the image of
    $\vec{L_3}$ has strictly more regular isotopy classes.
    In partiuclar, the  regular isotopy classes
    in the image of $\vec{L_3}$ are the same
    as the regular isotopy classes
    in the images of the unoriented mappings $L$ and $L_3$
    from Definitions~\ref{Ldef} and~\ref{L3def}.
\end{remark}

\begin{theorem}
    An oriented link diagram $D$ is compliant
    if and only if there is some $h \in \vec{F_3}$
    such that $\vec{L_3}(h) \simsub{23} D$.
\end{theorem}
\begin{proof}
    First suppose that $D$ is a link diagram with $D\simsub{23}\vec{L_3}(h)$ for some $h\in \vec{F_3}$.
    Then $D\simsub{23}L_3(h)$, so $D$ is compliant
    by Corollary~\ref{corollaryf3}.
    
    Now suppose $D$ is compliant.
    By Lemma~\ref{aiellolemma}, there is some $g \in \vec{F}$ such that $\vec{L}(g)\simsub{123} D$.
    There is a natural embedding $\vec{\phi}\colon \vec{F} \to \vec{F_3}$ such that $\vec{L_3} \circ \vec{\phi} = \vec{L}$, where supplementary arcs are re-imagined
    as parts of the tree. Put $h = \vec{\phi}(g)$,
    so then $\vec{L_3}(h) = \vec{L}(g) \simsub{123} D$.
    Each component of $\vec{L_3}(h)$ already has a leaf-arc
    by construction because
    every part of the diagram that crosses
    the horizontal axis (other than the closure strand) is part of a leaf-arc.
    By Lemma~\ref{cowardlemma}, it suffices to manipulate
    $h$ (replacing it by some other $h'\in \vec{F_3}$)
    so that for each component $C$ of $D$,
    the corresponding component $C'$ of $L(h')$
    has $\Omega(C) = \Omega(C')$ and $w(C) = w(C')$,
    while retaining $\vec{L_3}(h)\simsub{123} D$.
    We accomplish this by adjusting $\Omega$ and $w$
    for each component individually.
    
    Let $C$ be a component of $D$ and let $C'$ be the corresponding 
    component of $\vec{L_3}(h)$. Since $D$ and $\vec{L_3}(H)$
    are equivalent, and each is compliant, we have
    \[
        \alpha_D(C) = \beta_D(C) = \beta_{\vec{L_3}(h)} = \alpha_{\vec{L_3}(h)}(C').
    \]
    By Proposition~\ref{parityprop}, it follows that $w(C) - w(C')$ and $\Omega(C) - \Omega(C')$ are even. Thus, it suffices to repeatedly
    add $\pm 2$ to $w(C')$ and $\Omega(C')$ until they match
    $w(C)$ and $\Omega(C)$. To adjust $\Omega(C')$, we already have the moves
    from Figure~\ref{fig:oriented-insertions}, now regarding the dotted
    supplementary arcs as part of the tree. There is another complication
    in that when inserting such diagrams at some leaf in a ternary tree,
    the shaded face of the diagram can be on the right. In such cases,
    we must insert the reflections of the diagrams
    in Figure~\ref{fig:oriented-insertions} across a vertical axis.
    This swaps the effect on the change of index $\Omega$
    of the two moves, but a move is always available to increase or decrease $\Omega$
    of a given component by $2$.
    
    We use a similar process to adjust the writhe of each component,
    using the moves from Figure~\ref{fig:oriented-write-insertions}.
    Again, if the shaded face of the diagram is on the right,
    then we can reflect across the vertical axis and
    swap the effects of the two moves.
\begin{figure}
    \centering
    \includegraphics{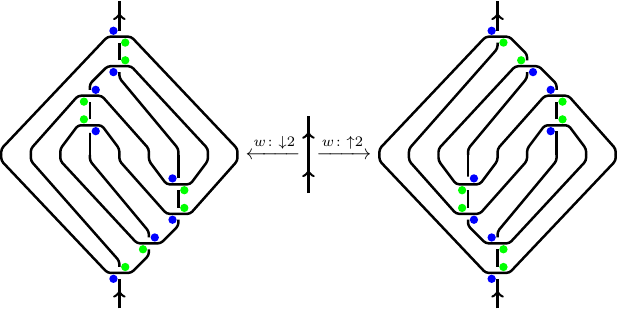}
    \caption{These two diagrams can be inserted at any leaf-arc of an existing diagram from $\vec{F_3}$ to get another oriented diagram with the writhe $w$ changed by $\pm 2$. If the leaf-arc is directed down instead of up, these same moves apply, but they have the inverse effects.}
    \label{fig:oriented-write-insertions}
\end{figure}
\end{proof}
}

\refchange{

\clearpage

\section{Conclusion}
In this paper, we classified the regular isotopy classes of link diagrams arising from elements of Thompson's groups $F, \vec{F}, F_3,$ and $\vec{F}_3$. For $F, F_3, $ and $\vec{F}_3$, the conclusions were the same: a link diagram $D$ is regular-isotopic to a link diagram arising from Thompson's group $(F, F_3, \text{ or }\vec{F}_3)$ if and only if $D$ is \textit{compliant}. In the case of $\vec{F}$, the classification was slightly different: a link diagram $D$ is regular-isotopic to a link diagram arising from $\vec{F}$ if and only if it is \textit{oriented-compliant}. In the case of $\vec{F}$, there was a pairing between crossings in the ``top half'' of the link diagram and the ``bottom half'' of the link diagram which corresponded positive crossings to negative crossings, forcing us to consider the function $u_D$ over $\mathbb{Z}$ instead of $\mathbb{Z}/2\mathbb{Z}$. There was no such phenomenon in $F$, $F_3$, or $\vec{F}_3$, allowing us to obtain more regular isotopy classes of link diagrams from each of those groups.

\begin{remark} One fundamental question about the construction of knots and links from elements of Thompson's groups, asked by Jones in \cite[Remark 5.3.3.2]{jones2014unitary}, is to prove an analogue of Markov's theorem for Thompson's group. In other words, the problem is to produce a set of transformations on elements of $F$ that preserve the equivalence class of the corresponding diagram, such that for any two $g,h\in F$ with $L(g)\simsub{123} L(h)$, we can transform $g$ to $h$ using these transformations. In this paper, we produce some such transformations, for example, in Figures \ref{fig:makeleaftree} and \ref{fig:oriented-makeleaf}. Our results also place limitations on these moves. For example, no transformation on elements of $F$ will perform a single Reidemeister I move on the corresponding link diagram.
\end{remark}
}

\section{Acknowledgements}
This work was done as part of the 2020 Knots and Graphs Working Group at The Ohio State University. We are grateful to the OSU Honors Program Research Fund and to the NSF-DMS \#1547357 RTG grant: Algebraic Topology and Its Applications for generous financial support. We are indebted to Sergei Chmutov and Matthew Harper for their guidance and helpful suggestions throughout. Special thanks to Jake Huryn for many helpful conversations and an invaluable contribution to Lemma \ref{leaflemma}, and to Vaughan Jones for suggesting improvements to the paper.

\bibliographystyle{plain}
\bibliography{references}

\begin{thebibliography}{1}

\bibitem{aiello2020}
Valeriano Aiello.
\newblock On the {A}lexander theorem for the oriented {T}hompson group {$F$}.
\newblock {\em Algebraic \& Geometric Topology}, 20(1):429–438, February
  2020.

\bibitem{CFP}
J.~W. Cannon, W.~Floyd, and W.~Parry.
\newblock Introductory notes on {R}ichard {T}hompson's groups.
\newblock {\em L'Enseignement Mathematique}, 42:215--256, 1996.

\bibitem{coward}
Alexander Coward.
\newblock Ordering the {R}eidemeister moves of a classical knot.
\newblock {\em Algebraic \& Geometric Topology}, 6(2):659–671, May 2006.

\bibitem{Kauffman}
Louis {H. Kauffman}.
\newblock State models and the {J}ones polynomial.
\newblock {\em Topology}, 26(3):395 -- 407, 1987.

\bibitem{jones2014unitary}
Vaughan F.~R. Jones.
\newblock Some unitary representations of {T}hompson's groups {F} and {T},
  2014.

\bibitem{jones2018construction}
Vaughan F.~R. Jones.
\newblock On the construction of knots and links from {T}hompson's groups.
\newblock In {\em Knots, Low-Dimensional Topology and Applications}, pages
  43--66. Springer International Publishing, 2019.

\bibitem{KauffmanInvariant}
Louis~H. Kauffman.
\newblock An invariant of regular isotopy.
\newblock {\em Transactions of the American Mathematical Society},
  318(2):417--471, 1990.

\bibitem{Reidemeister}
Kurt Reidemeister.
\newblock Elementare begr{\"u}ndung der knotentheorie.
\newblock {\em Abhandlungen aus dem Mathematischen Seminar der Universit{\"a}t
  Hamburg}, 5(1):24--32, December 1927.

\bibitem{whitney}
Hassler Whitney.
\newblock On regular closed curves in the plane.
\newblock {\em Compositio Mathematica}, 4:276--284, 1937.

\end{thebibliography}

\end{document}